\documentclass{birkjour}
\usepackage{epstopdf}
\usepackage[caption=false]{subfig}

\usepackage[numbers,sort&compress]{natbib}
\bibpunct[, ]{[}{]}{,}{n}{,}{,}
\makeatletter
\def\NAT@def@citea{\def\@citea{\NAT@separator}}
\makeatother

\theoremstyle{plain}
\newtheorem{theorem}{Theorem}[section]
\newtheorem{lemma}[theorem]{Lemma}
\newtheorem{corollary}[theorem]{Corollary}
\newtheorem{proposition}[theorem]{Proposition}
\newtheorem{definition}[theorem]{Definition}
\newtheorem{example}[theorem]{Example}
\newtheorem{remark}[theorem]{Remark}

\usepackage{comment}

\newcommand{\dg}{{\dagger}}

\newcommand{\R}{{\mathcal{R}}}
\newcommand{\N}{{\mathbb{N}}}

\newcommand{\ep}{\scriptsize\mbox{\textcircled{$\dagger$}}}
\newcommand{\core}{\scriptsize\mbox{\textcircled{\#}}}
\newcommand{\wc}{\boxplus}
\newcommand{\fpower}[2]{{}^{#1} \kern -0.1em #2}
\newcommand{\cd}{\textcircled{d}}

\newcommand{\cc}{\boxminus}
\newcommand{\tp}{\textup}
\newcommand{\ind}{\mbox{\normalfont ind}}
\newcommand{\ii}{\mbox{ \normalfont i}}

\usepackage{amsmath,amsthm,amsfonts,amssymb,enumerate}
\usepackage{paralist}
\usepackage{tabto}
\usepackage{textcomp}
\usepackage[utf8]{inputenc}
\usepackage{tikz,xcolor,hyperref}
 \numberwithin{equation}{section}

\usepackage{epstopdf,enumerate}
\usepackage{comment}
\usepackage[active,new,old]{correct} 

\newcommand{\tr}{{\mbox{tr}}}

\usepackage{calc}
\usepackage{centernot}
\usepackage{mathtools}
\usepackage{ stmaryrd }
\usepackage{amsmath,amsthm,amsfonts,amssymb,enumerate}
\usepackage{paralist}
\usepackage{tabto}
\usepackage{textcomp}
\usepackage[utf8]{inputenc}
\usepackage{tikz,xcolor,hyperref}

\definecolor{lime}{HTML}{A6CE39}
\definecolor{lightblue}{rgb}{0.0, 0.0, 0.5}
\DeclareRobustCommand{\orcidicon}{%
	\begin{tikzpicture}
	\draw[lime, fill=lime] (0,0)
	circle [radius=0.16]
	node[white] {{\fontfamily{qag}\selectfont \tiny ID}};
	\draw[white, fill=white] (-0.0625,0.095)
	circle [radius=0.007];
	\end{tikzpicture}
	\hspace{-2mm}
}

\foreach \x in {A, ..., Z}{%
	\expandafter\xdef\csname orcid\x\endcsname{\noexpand\href{https://orcid.org/\csname orcidauthor\x\endcsname}{\noexpand\orcidicon}}
}

\begin{document}

%
%
%
%
%
%
%
%
%

\title[{
Generalized Core Inverse in a {\it proper $*$-ring}}]{
Generalized Core Inverse in a proper $*$-ring}

\author[R. Behera]{Ratikanta Behera
\orcidB
}
\address{ Department of Computational and Data Sciences\\
Indian Institute of Science, Bangalore, 560012, India}
\email{ratikanta@iisc.ac.in}

\author[J. K. Sahoo]{Jajati Keshari Sahoo
\orcidA
}

\address{%
Department of Mathematics,\\  Birla Institute of Technology $\&$ Science,\\ Pilani K. K. Birla Goa Campus, Goa, India}

\email{jksahoo@goa.bits-pilani.ac.in}



\author[R. N. Mohapatra, ]{R. N. Mohapatra
\orcidC
}
\address{Department of Mathematics,\\ University of Central Florida,\\ Orlando, USA}
\email{ram.Mohapatra@ucf.edu}
\author[S. Das]{Sourav Das
\orcidD
}
\address{Department of Mathematics,\\ National Institute of Technology Jamshedpur,\\ Jharkhand-831014, India}
\email{souravdas.math@nitjsr.ac.in}
\author[S. K. Prajapati]{Sunil Kumar Prajapati}
\address{Department of Mathematics, School of Basic Sciences\\ IIT Bhubaneswar,\\ Bhubaneswar, India}
\email{skprajapati@iitbbs.ac.in}


\date{}

\begin{abstract}
In this paper, we introduce the notion of weak core and central weak core inverse in a {\it proper $*$-ring}. We further elaborate on these two classes by producing a few representations and characterizations of the weak core and central weak core invertible elements.  We investigated additive properties and a few explicit expressions for these two classes of inverses through other generalized inverses. In addition, numerical examples are provided to validate  claims on weak core inverses. 
 Following {\it proper $*$-ring} and their interconnections with Clifford algebra, we also present examples of the group inverse and the weak core inverse of a non-zero non-invertible quaternion $\mathbb{H}_s$.
\end{abstract}

\maketitle

\section{Introduction}
Let $\R$ be a ring with unit $1 \neq 0$ and involution $r \mapsto r^*$ satisfying $(r^*)^*=r,~ (r + s)^*=r^*+s^*,$ and $(rs)^*=s^*r^*$ for $r,s\in\R.$
 Ring $\R$ is called a {\it proper $*$-ring} if $r^*r = 0$ implies $r=0$ for an arbitrary element $r \in \R$, which is defined in \cite{weakg}. However, the authors of \cite{koliha} called this $*$-reducing. The notion of the core inverse on an arbitrary $*$-ring was introduced in \cite{Rakietal2014} and has been investigated over the past few years. However, the authors of \cite{baks, BakTr14}  introduced the concept of the core inverse for matrices. The Drazin inverse was introduced in \cite{Drazin58} for rings and semigroups. Many researchers \cite{Drazin13,Drazin13left} explored numerous properties of the Drazin inverse and interconnections with other generalized inverses. It is worth mentioning that the spectral properties \cite{cline1968} of the Drazin inverse play a significant role in many applications \cite{ben}. In reference to the theory of finite Markov chains,  Meyer in \cite{meyer} studied the advantages of the Drazin inverses with respect to the inner inverse and the Moore-Penrose inverse. The author also observed that the computations used in the Moore-Penrose inverse can be unstable. Several representations and properties of the core invertible elements in $*$-ring are discussed in \cite{Ma2018,XuJ2017}. However, the weak Drazin inverse of 
  matrices was discussed in \cite{Cambel78} for studying special kinds of systems of differential equations. Then, Wang and Chen \cite{Wng2018} introduced the weak group inverse for complex matrices. In this connection, Zhou et al.~\cite{weakg} discussed the notion of the weak group inverse in a {\it proper $*$-rings}. It will be more applicable if we introduce weak core and central weak core inverse in a {\it proper $*$-rings}; hence, these inverses will supply the necessary freedom to deal with different types of generalized inverses. This provides the flexibility to choose generalized inverses depending on the applications.

The main idea of the central Drazin inverse comes from the commuting properties of generalized inverses (see \cite[Example 2.8]{Drazin13}). Following this subclass of the Drazin-invertible elements, Wu and Zhao \cite{Centrald19} discussed a few characterizations of central Drazin-invertible elements in a ring. The authors of \cite{zhao2020one}, further discussed one-sided central Drazin-invertible elements in a ring. The vast literature on the core inverses in $*$-rings and multifarious extensions along with subclasses of the Drazin inverse \cite{Cambel78, zhao2020one}, and group inverse \cite{Wng2018},  motivated us to introduce weak core and central weak core inverse in a {\it proper $*$-ring}.

The main contributions of this paper are listed in the following points.

\begin{enumerate}
    \item[$\bullet$] Weak core and central weak core inverse are introduced in a proper $*$-ring.
    \item[$\bullet$] A few explicit expressions for the weak core and central weak core inverse in a proper $*$-ring through other generalized inverses such as the Drazin inverse, core inverse, and Moore-Penrose inverse, are discussed.
    \item[$\bullet$] Several characterizations and representations of these two classes of the inverses are established.
    \item[$\bullet$] Additive properties for weak core inverse and central weak core inverse are presented.
\end{enumerate}

\begin{figure}[!htb]
\centering
\includegraphics[width=.7 \columnwidth]{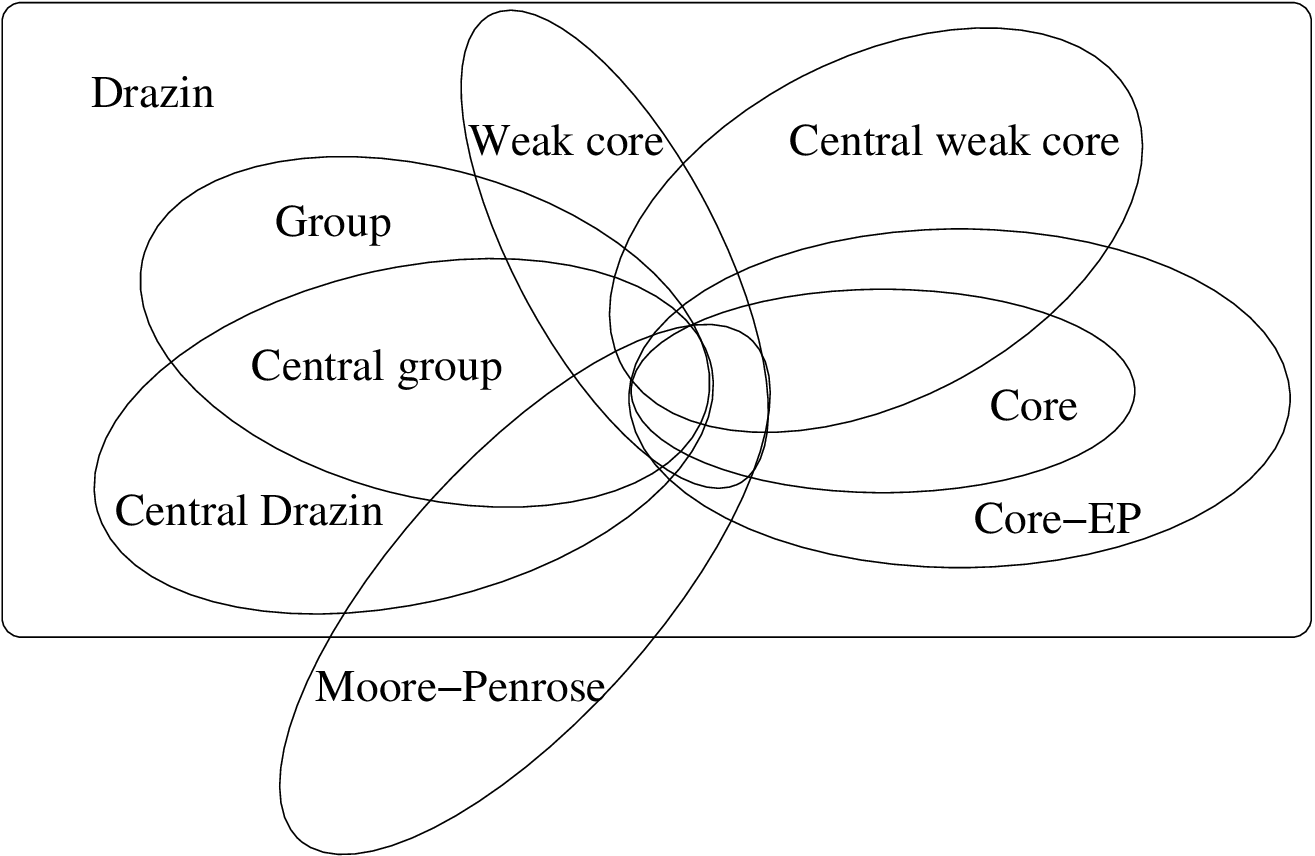}
\caption{Structural representation of different generalized inverses}
\end{figure}\label{FiguGI}

The various kind of generalized inverses and their relations are demonstrated in Figure \textbf{1}. A large amount of work has already been devoted to the Moore-Penrose \cite{koliha, Rakietal2014, mosic17}, the Drazin \cite{Drazin58}, core \cite{baks,BakTr14}, core-EP \cite{ManMo14,MaPedrag19, gao18} invertible elements in a ring.  The purpose of this paper is to propose two classes of core inverses, i.e., weak core inverse and central weak core inverse. We investigate the properties of these two classes of inverses and their relationships with other generalized inverses.  The major strength of these classes is that it can be applied easily to $C^*$-algebra (see Koliha et~al.~\cite{koliha} for the Moore–Penrose inverse).

However, the problem of the sum of two generalized invertible elements in $*$-ring generated a tremendous amount of interest in the algebraic structure of ring theory \cite{IliBook17,Gonz_lez_2004,Djordjevi__2002, Zhou-CA-2020}. In this context, Moore \cite{Penrose55} first discussed invertible elements in a complex matrix ring.  Since then, many researchers have studied the additive properties of various classes of generalized inverses in \cite{CZW, Cvetkovi_Ili__2006, WeiPe18, Gonz_lez_2004, SahBe20}. In the study, we derive an explicit expression for weak core and central weak core invertible elements in a proper $*$-ring.

\subsection{Relation with Clifford algebras $C\ell_{p,q}$}
It is well known that the Clifford algebra is associated with a vector space with an inner product. Real Clifford algebras $C\ell_{p,q}$ over a quadratic real vector space $(V,Q)$ of
dimension $n = p + q$ and a non-degenerate quadratic form $Q$ of signature $(p; q)$  possess many involutions such as reversion $\beta$ or Clifford conjugation $\gamma$ \cite{b6,b7}.
In the algebras $C\ell_{p,0}$, the reversion $\beta$ is a positive involution in that it defines on the vector space of the algebra a positive definite bilinear form $\omega$ as
$\omega(x, y) = \tr(\beta(x)y) = \mbox{scal}(\beta(x)y)$ where scal is the scalar part (or, the 0 -part) of the product $\beta(x)y$.

In this context of the algebras $C\ell_{0,q}$, the Clifford conjugation is a positive involution in the same sense: the bilinear form $\omega(x; y) = \tr(\gamma(x)y) = \mbox{scal}(\gamma(x)y)$ is positive definite. Then, in every Clifford algebra $C\ell_{p,q}$ there is a transposition involution $\tau$ which defines a positive definite bilinear form $\omega(x; y) = \tr( \tau(x)y) = \mbox{scal}( \tau(x)y)).$  In fact, $\beta=\tau$
in signatures $(p, 0)$ and $\gamma=\tau$  in signatures $(0, q)$. Furthermore, depending on the value of $p-q$ mod $8$,  involution $\tau$ gives matrix transposition, complex Hermitian conjugation, or quaternionic Hermitian conjugation of matrices in a spinor representation of the algebra \cite{b2,b3,b4}. Thus, every Clifford algebra $C\ell_{p,q}$ is a {\it proper $\tau$-ring}
Hence, applying these new concepts of generalized inverses to the Clifford algebras $C{\ell}_{p,q}$ would be interesting.

\subsection{Outlines}
The remainder of this paper is organized as follows. In Section 2,  we discuss some useful notations and definitions along with a few essential preliminary results. The weak core inverse and its characterization are established in Section 3. In Section 4, we discuss the central weak core inverse and its relationship with other generalized inverses. The paper is summarized in Section 5, along with a few future perspectives for weak core and central weak core inverses.

\section{Preliminaries}
Throughout this study, we use the notations $a\R = \{az~:~z\in \R\}$
and ${\R a} = \{za~:~ z\in\R\}$.  The center of $\R$ is denoted by $C(\R).$ The right annihilator of $a$ is defined by $\{z\in\R~: az = 0\}$ and is denoted as $a\fpower{\circ}$. Similarly, the left annihilator of $a$ is defined by $\{z\in \R~: za = 0\}$ and is denoted as $\fpower{\circ}{a}$.  {We consider $\mbox{Mat}(2, R)$ as the algebra of $2 \times 2$ real matrices and recall the split quaternions \cite{b1},  $\mathbb{H}_s=\{q=q_0+q_1 {\bf i}  +q_2 {\bf j} +q_3 {\bf k} \;|\; q_0,q_1,q_2,q_3\in \mathbb{R}\}$ with ${\bf i}^2=-{\bf j}^2=-{\bf k}^2=-{\bf i j k}=-1$ this yield $-{\bf ji}={\bf ij}=\textbf{k}$, $-{\bf kj}={\bf jk}=-{\bf i}$, $-{\bf ik}={\bf ki}=j$.  It is well-known that $\mathbb{H}_s\cong C\ell_{1,1}\cong \mbox{Mat} (2,\mathbb{R})$, i.e.
\begin{align}\nonumber
\mathbb{H}_s & \overset{\psi}\cong C\ell_{1,1} \overset{\theta}\cong \mbox{Mat} (2,\mathbb{R}),\\\nonumber
 q & = q_0+q_1 {\bf i}  +q_2 {\bf j} +q_3 {\bf k}\xmapsto{\psi} \; q=q_0+q_1 e_2 +q_2 e_1+q_3 e_2 e_1 \\\label{Qmatrix}
&\xmapsto{\theta} Q
= \left[  \begin{array}{c c}
q_0-q_3&q_2-q_1\\
q_2+q_1&q_0+q_3
\end{array} \right]
\end{align}
where  $e_1,$ and $e_2$ are orthonormal vectors generates the Clifford algebra $C\ell_{1,1}$ that is
$e_1^2=1$, $e_2^2=-1$, $e_1e_2=-e_2 e_1$. We now define
\begin{equation*}
\theta \circ \psi = \varphi
: \mathbb{H}_s\rightarrow \mbox{Mat}(2,\mathbb{R}).
\end{equation*}
Following \eqref{Qmatrix} with $Q\mapsto Q^{T}$, we define the involutive anti-automorphism of
$\mathbb{H}_s$ as follows:
\begin{align}\nonumber
&\tau: \mathbb{H}_s \rightarrow \mathbb{H}_s, \\\nonumber
& q= q_0+ q_1 {\bf i}+ q_2 {\bf j} + q_3 {\bf k}
\mapsto \tau(q) = q_0- q_1 {\bf i}+ q_2 {\bf j} + q_3 {\bf k}, \\
&  Q 
\mapsto Q^{T}= \left[  \begin{array}{c c}
q_0-q_3 & q_2+q_1\\
q_2-q_1 & q_0+q_3
\end{array} \right].
\end{align}
Hence, $\tau (\tau(q))=q$ and $\tau(ab)=\tau(b)\tau(a)$ for any quaternion $q,a,b\in\mathbb{H}_s$.
The conjugate $\bar{q}$ of the quaternion $q$ is defined below.
\begin{align}
q\mapsto \bar{q}=q_0 - q_1 {\bf i} - q_2 {\bf j} - q_3 {\bf k}.
\end{align}
Now, we have $q\bar{q}=\bar{q}q=q_0^2+q_1^2-q_2^2-q_3^2$. It can be verified that $C\ell_{1,1}$ is the 
quaternionic conjugation. Hence,
\begin{align}
q\bar{q} = \psi (q) \alpha(\beta(\psi(q))) =\det (Q)~~\textnormal{for any quaternion $q$.}
\end{align}
}

%
%
Let us recall the definition of the Moore-Penrose \cite{koliha}, core \cite{baks,BakTr14}, core-EP \cite{ManMo14, gao18}, and Drazin \cite{Drazin58} inverse of an element in $\R$.
 \begin{definition}\label{defgi}
For any element $a \in \R,$  consider the following equations in $z \in \R:$
\begin{eqnarray*}
&&(1)~aza=a,\quad(2)~ zaz=z,\quad(3)~(az)^* =az,\quad(4)~(za)^*=za,\\
&& (5)~az=za,\quad(6)~za^2=a,\quad(6^k)~za^{k+1}=a^k,\quad(7)~az^2=z.
\end{eqnarray*}
Then $z$ is called
\begin{enumerate}[\rm(a)]
\item generalized (or inner) inverse of $a$ if it satisfies $(1)$ and is denoted by $a^{(1)}.$
\item   $\{1,3\}$ inverse of $a$ if it satisfies $(1)$ and $(3)$, which is denoted by $a^{(1,3)}.$
\item  the Moore-Penrose inverse of $a$ if it satisfies all four conditions $(1)-(4)$, which is denoted by $a^{\dagger}.$
\item {the Drazin inverse of $a$ if it satisfies conditions $(2),~(5)$, and $(6^k)$, and is denoted by $a^{D}$. Then, the smallest positive integer $k$ for which the conditions are true is called the index (Drazin index) of $a$, and  is denoted by $\ii(a)$. In particular, when $k=1$, we refer to $z$ as the group inverse of $a$, and we denote it by $a^{\#}$.}
\item {the core-EP inverse of $a$ if it satisfies the conditions $(3),~(6^k)$, and $(7)$, and  is denoted by $a^{\textup{\ep}}$. For $k=1$, we refer to $z$ as the core inverse of $a$, and we denote it by $a^{\textup{\core}}$.}
\end{enumerate}
\end{definition}

{
The authors of \cite{koliha} discussed the uniqueness (when it exists) of the Moore-Penrose inverse. The uniqueness of the group inverse, Drazin inverse, and core inverse was studied \cite{Rakietal2014,Drazin1958} in rings with involution. Further, Gao and Chen \cite{GaoChen18} proposed the core-EP inverse and showed that it is unique when it exists.

In the recent papers \cite{b1, b5}, the Moore-Penrose inverses and pseudo-inverses were studied without matrices in associative unital rings endowed with a positive involution. Then, the transposition involution $\tau$ in $C\ell_{p,q}$ maps every basis element $ e_{i_1}e_{i_2}\cdots e_{i_r}$ ($i_1<i_2<\cdots<i_r$ and $1\leq r\leq n$) to its inverse
$(e_{i_1}e_{i_2}\cdots e_{i_r})^{-1}$ under the assumption that $e_1e_2\cdots e_n$ is an orthonormal basis in $(V,Q)$. This is because $\tau$ is the same as involution $\ast: K[G] \rightarrow K[G]$ in a group algebra $K[G]$  over a finite group $G$ and a  field $K$ defined in \cite{b8}. 
We refer the reader to the explanation in \cite{Rafal18, Rafel18Spi, Sali84} for details.}

{An example of the Moore-Penrose inverse of a non-zero non-invertible quaternion $\mathbb{H}_s$ is discussed in \cite{b1}, which is recalled next, and we present an example of the group inverse of a non-zero non-invertible quaternion $\mathbb{H}_s$.}

\begin{example}[Example 1, \cite{b1}]\rm\label{exa2.2}
{Consider $\R=\mathbb{H}_s$, be the algebra of split quaternions with involutive anti-automorphism $\tau$. Let $a=1+2\textbf{i}+\textbf{j}+2\textbf{k}$. Clearly, element $a$ is a non-invertible quaternion in $\mathbb{H}_s$. In \cite{b1}, it was found that the Moore-Penrose inverse of $a$ is, $a^{\dagger}=\frac{1}{20}(1-2\textbf{i}+\textbf{j}+2\textbf{k})$. We can see that the quaternion $x=\frac{1}{4}(1-2\textbf{i}+\textbf{j}+2\textbf{k})$ satisfies $ax a=a$, $xa x =x$, and $ax=x a$. Thus $a^{\#}=x=\frac{1}{4}(1-2\textbf{i}+\textbf{j}+2\textbf{k}).$ }
\end{example}
For convenience, we use  $\R^{\dagger}$, $\mathcal{R}^{\#}$, $\mathcal{R}^{\core}$, $\mathcal{R}^{D}$, $\mathcal{R}^{\ep}$ respectively for the set of all Moore-Penrose, group, core, Drazin,  and core-EP invertible elements of $\mathcal{R}$. Next, we present a few auxiliary results that are essential to prove some of our results.

\begin{lemma}\label{lm:core-drazin}{\rm\cite{mosic18}}
{Let $a\in\R$, and both $a^\dagger$, $a^{\textup{\#}}$ exist}, then $a^{\textup{\core}}=a^{\#}aa^{\dg}=aa^{\#}a^{\dg}$.
\end{lemma}

\begin{proposition}\label{prop2.3}{\rm\cite{gao18}}
{Let $a\in\R$ be Drazin-invertible with $\ii(a) = k$ and let $(a^k )^{(1,3)}$ exist. Then,} $a$ is core-EP invertible  and $a^{\tp{\ep}}=a^Da^k(a^k)^{(1,3)}$.
\end{proposition}
\begin{lemma}\label{prop:sum-Drazin} {\rm\cite{gao18}}
Let $a,b\in\R$ be Drazin-invertible elements with $ab=0=ba$. Then  $(a+b)^{D}=a^{D}+b^{D}.$
\end{lemma}

\begin{proposition}\label{prop2.2}{\rm\cite{XuJ2017}}
Let $a\in \R.$ If an element $z\in\R$
satisfies $aza = a~, az^2=z$ and $(az)^{*}=az$, then ${a^{\tp{\core}}=z}$.
\end{proposition}

\begin{lemma}\label{lm:Drazin}
 Let $a\in\R$. {If there exists $y\in\R$ satisfying the following relations:}
$$
ay^2=y~~\mbox{ and } ~~ya^{k+1}=a^k \mbox{ for some positive integer }k, ~\mbox{ then }
$$
\begin{enumerate}[\rm(i)]
\item $ay=a^{m}y^{m}$ for any positive integer $m$;
\item $yay=y$;
\item $a$ is Drazin-invertible with $a^{D}=y^{k+1}a^k,$ and $\ii(a)\leq k$;
\item $a^my^ma^m=a^m$ for $m\geq k$;
\item $y\R= a^k\R$.
\end{enumerate}
\end{lemma}
\begin{proof}
The proof of parts (i)-(iv) can be found in  \cite{gao2018}. From $y=yay=ya^{k+1}y^{k+1}=a^ky^{k+1}$ and $a^{k}=ya^{k+1}$, we obtain $y\R=a^k\R.$
\end{proof}

\begin{lemma}\label{lm:o}{\rm\cite{von_Neumann_1936}}
Let $b,c\in\R$. Then the following assertions hold:
\begin{enumerate}[\rm(i)]
\item If $b\R\subseteq c\R$, then $\fpower{\circ}{c}\subseteq \fpower{\circ}{b}$;
\item If $\R b\subseteq \R c,$ then $c\fpower{\circ}\subseteq b\fpower{\circ}$.
\end{enumerate}
\end{lemma}

\begin{proof}
{(i). From $b\R\subseteq c\R$, we have $b=ct$ for some $t\in \R$. Let $z\in \fpower{\circ}{c}$. Then $zc=0$ and subsequently $zb=zct=0$. Thus $z\in \fpower{\circ}{b}$. Hence $\fpower{\circ}{c}\subseteq \fpower{\circ}{b}$.\\
(ii) Let  $\R b\subseteq \R c$. Then $b=rc$ for some $r\in \R$. For $x\in c\fpower{\circ}$, we have $cx=0$. This result implies $bx=rcx=0$. Thus $c\fpower{\circ}\subseteq b\fpower{\circ}$.
}
\end{proof}

Hereafter, $\R$ is assumed to be a proper $*$-ring. Next, we define the EP element in a ring with involution.

\begin{definition} \cite{Koliha2002}   
An element $a\in\R$
is called EP if $a\in\R^{\#}\cap\R^{\dagger}$ and $a^{\#}=a^{\dagger}$.
\end{definition}

The relationship between the Drazin inverse and group inverse (which was given in \cite{BenIsrael03} for matrices) is presented below.

\begin{proposition}\label{lm:drazin-group}
Let $a\in\R^{D}$ with $\ii(a)=k.$ Then $\left(a^m \right)^{\#}=\left(a^{D} \right)^m$ for all $m\geq k.$
\end{proposition}

\begin{proof}
{For $m\geq k$, let $z=(a^D)^m$. Now
\begin{eqnarray*}
a^mza^m&=&a^m(a^D)^ma^m=a^m(a^D)^{m-1}aa^Daa^{m-2}=a^m(a^D)^{m-1}a^{m-1}\\
&=&\cdots =a^ma^Da=a^Da^{m+1}=a^m,
\end{eqnarray*}
$za^mz=(a^D)^ma^m(a^D)^m=a^Da(A^D)^m=(a^D)^m$, and $a^mz=za^m$
is trivial because $aa^D=a^Da$. Thus $(a^m)^{\#}=z=(a^D)^m$.
}
\end{proof}

We now recall the weak group inverse \cite{weakg} of $a$  in a proper $*$-ring.

\begin{definition}{\rm\cite{weakg}}
Let $a\in\R$. Then an element $y\in\R$ is called the weak group inverse of $a$ if it satisfies
\begin{equation}\label{eqwg}
 { ya^{k+1} = a^k,~ ~ay^2 = y, \mbox{ and}~(a^k)^*a^2y = (a^k)^*a}.
\end{equation}
The smallest positive integer $k$ for which \eqref{eqwg} holds, is called the index of $a$ (weak group index) and is denoted by $\ind_{wg}(a)$. The weak group inverse of an element is represented by $a^{\textup{\textcircled{w}}}$ and the set of weak group invertible elements is denoted by $\R^{\textup{\textcircled{w}}}$.
\end{definition}

The relationship between the group inverse and weak group inverse is discussed in \cite{weakg} and presented below.

{
\begin{remark}[Remark 3.2, \rm\cite{weakg}]\label{prop2.7}
If $a\in \R^{\#}$, then a is weak group
invertible and $a^{\textup{\textcircled{w}}}=a^{\#}$.
\end{remark}
}
Next, we recall the definition {of a central Drazin} inverse of an element.
\begin{definition}\label{eqcdraz}{\rm \cite{Centrald19}}
Let $a\in\R$. An element $y\in\R$ satisfying
\begin{equation}\label{cdraz}
    ya \in C(\R),~~ yay= y, \mbox{~~and}~ a^{k+1}y = a^k,
\end{equation}
is called central Drazin inverse of $a$ and denoted by $a^{\textup{\textcircled{d}}}$.
\end{definition}
 If such $y$ exists, then $a$ is called central Drazin-invertible. The smallest $k$ for which \eqref{cdraz} holds, is {called an index} (central Drazin index) of $a$ and {it is denoted} by $\ind_{cd}(a)$. The set of central Drazin elements is denoted by $\R^{\tp{\textcircled{d}}}$. {When $k=1$,  $y$ is referred as the central  group inverse of $a$.} We collected the following useful {results on the central Drazin inverse.}
 \begin{proposition}\label{prop2.11}\rm\cite{Centrald19}
 Let $a \in \R^{\tp{\textcircled{d}}}$ and $x=a^{\textup{\textcircled{d}}}$. Then the following assertions hold:
\begin{enumerate}[\rm(i)]
    \item  $ax=xa$;
    \item  $a^nx^n=ax$ for any positive integer $n$.
\end{enumerate}
\end{proposition}

\begin{lemma}\label{lm:central-D-group}
Let $a\in\R$ with $\ii(a)=k.$ Then $\left(a^m \right)^{\#}=\left(a^{\tp{\textcircled{d}}} \right)^m$ for all $m\geq k.$
\end{lemma}
\begin{proof}
Using Proposition \ref{prop2.11}, we have
\begin{align*}
a^m \left(a^{\tp{\cd}}\right)^m a^m&=a a^{\tp{\cd}} a^m=a^{\tp{\cd}} a^{m+1}=a^m,\\
\left(a^{\tp{\cd}}\right)^m a^m \left(a^{\tp{\cd}}\right)^m &= \left(a^{\tp{\cd}}\right)^m a^m a^{\tp{\cd}}=\left(a^{\tp{\cd}}\right)^m,
\end{align*}
and $a^m(a^{\tp{\cd}})^m=(a^{\tp{\cd}})^ma^m$. Hence $\left(a^m \right)^{\#}=\left(a^{\tp{\textcircled{d}}} \right)^m$ for all $m\geq k.$
\end{proof}

\section{Weak core inverse}
{In this section, we define the weak core inverse of an element in a proper $*$-ring.  We then generalize the concept of a weak group inverse to a weak core inverse, which we define as follows.}
\begin{definition}\label{def-wc}
Let $a\in\R$. {An element $y\in\R$  is called the weak core inverse of $a$ if it satisfies the following  three conditions:
$$
(6^k)~~ya^{k+1}=a^k,~~(7)~~ay^2=y,~~(6^*)~~\left(a^k\right)^{*}ay=\left(a^k\right)^{*},
$$
and is denoted as $a^{\wc}$. The smallest positive integer $k$ that satisfies $(6^k)$, $(7)$  and $(6^*)$, is called the index (weak core index) of $a$ and it is denoted by $\ind_{wc}(a)$. If such $y$ exists, then $a$ is said to be a weak core invertible. The set of weak core invertible elements of $\R$ is denoted by $\R^{\wc}$.}
\end{definition}

Next, we present an example of the weak core inverse of a non-zero non-invertible quaternion $\mathbb{H}_s$.

\begin{example}\rm
{Let $\R=\mathbb{H}_s$ and  $a=1+2\textbf{i}+\textbf{j}+2\textbf{k}$. We can verify that the quaternion $y=\frac{1}{20}(5-3\textbf{j}+4\textbf{k})$ satisfies $ya^2=a$, $ay^2=y$, and $\tau(a)ay=\tau(a)$. Hence $a^{\wc}=y=\frac{1}{20}(5-3\textbf{j}+4\textbf{k})$ with index 1.}
\end{example}

The uniqueness {of a weak core} inverse is proved in the following result.
\begin{proposition}
Let $a\in \R^{\wc}$. Then the weak core inverse of $a$ is unique.
\end{proposition}
\begin{proof}
Let $x$ and $y$ be two {weak  core} inverses of $a$. Using Definition \ref{def-wc}
and Lemma \ref{lm:Drazin}, we obtain
\begin{align*}
{(ax)^{*}(ay)=\left(a^{k}x^k{k}\right)^{*}ay = \left(x^{k}\right)^{*} \left(a^{k}\right)^{*} a y= \left(x^{k}\right)^{*} \left(a^{k}\right)^{*}= \left(a^{k}x^{k}\right)^{*}
=(ax)^{*}.}
\end{align*}
Similarly, we have
$$
(ax)^{*}(ax)=(ax)^{*},~~(ay)^{*}(ax)=(ay)^{*}~~\mbox{ and }~~ (ay)^{*}(ay)=(ay)^{*}.
$$
Let $z=ax-ay.$ Then we obtain
\begin{align*}
z^{*}z&=(ax-ay)^{*}(ax-ay)
=(ax)^{*}ax-(ax)^{*}(ay)-(ay)^{*}(ax)+(ay)^{*}(ay)\\
&=(ax)^{*}-(ax)^{*}-(ay)^{*}+(ay)^{*}=0.
\end{align*}
Therefore $z=0$ and hence $ax=ay$. Using $ax=ay$ and Lemma \ref{lm:Drazin}, we obtain
\begin{equation*}
\begin{split}
{x=xax=xa^{k+1}x^{k+1}=a^{k} x^{k+1}=ya^{k+1}x^{k+1}=yax=yay=y.}
\qedhere
\end{split}
\end{equation*}
\end{proof}
Next, we establish a few characterizations {of the weak core inverse.}
\begin{theorem}
Let $a\in\R$. {The following assertions} are equivalent:
\begin{enumerate}[\rm(i)]
\item $y=a^{\wc}$ and  $\ind_{wc}(a)\leq k.$
\item $y=yay,$ $y\R=a^k \R=a^{k+1}\R$ and $a^k\R \subseteq y^{*}\R.$
\item $y=yay,$ $y\R=a^k \R \subseteq a^{k+1}\R$ and $\fpower{\circ}\left(y^*\right)\subseteq \fpower{\circ}\left(a^k\right).$
\item $y=yay,$ $\fpower{\circ}\left(a^{k+1}\right)\subseteq \fpower{\circ}\left(a^k\right)=\fpower{\circ}y$ and $\fpower{\circ}\left(y^*\right)\subseteq \fpower{\circ}\left(a^k\right).$
\end{enumerate}
\end{theorem}
\begin{proof}
(i) $\Rightarrow$ (ii):\hspace{0.3cm} Let $y=a^{\wc}$ and  $\ind(a)\leq k.$ Using Lemma \ref{lm:Drazin}, we obtain
$y=yay$ and $y\R=a^k\R$. Since $a^{k+1}\R\subseteq a^k \R$ and  $y=ay^2=a^k y^{k+1}=a^{k+1}y^{k+2}$, {it follows} that
$y\R\subseteq a^{k+1}\R\subseteq a^k\R=y\R$. {Furthermore,} this implies that, $y\R=a^{k+1}\R$. From  $\left(a^k\right)^{*}ay=\left(a^k\right)^{*}$, we have
 $y^*a^* a^k =a^k$. Consequently, $a^k\R\subseteq y^{*}\R.$

Using Lemma \ref{lm:o}, we can easily prove that (ii) $\Rightarrow$ (iii) $\Rightarrow$ (iv).

(iv) $\Rightarrow$ (i):\hspace{0.3cm} {Let $yay=y$. Then  $(ya-1)\in \fpower{\circ}y=\fpower{\circ}\left( a^k\right)$. This implies $(ya-1)a^k=0,$ that is, $ya^{k+1}=a^k.$ Furthermore,  $aya^{k+1}=a^{k+1}.$ Thus $(ay-1)\in \fpower{\circ}\left(a^{k+1}\right)\subseteq \fpower{\circ}y.$ Hence,  $(ay-1)y=0$ which is equivalently $ay^2=y.$}
Again, $\left(y^*a^*-1\right)\in \fpower{\circ}\left( y^*\right)\subseteq \fpower{\circ}\left( a^k\right)$  implies $y^*a^*a^k=a^k.$ Thus,
$\left(a^k\right)^{*}ay=\left(a^k\right)^{*}.$ Therefore, $y=a^{\wc}$ and $\ind(a)\leq k.$
\end{proof}
The construction of weak core inverse by using inner inverse is presented below.
\begin{theorem}
Let $a\in\R$. {The following assertions} are equivalent:
\begin{enumerate}[\rm(i)]
\item $a\in\R^{\wc}$ and $\ind_{wc}(a)\leq m.$
\item There  {is an idempotent} $p\in\R$ such that
$a^m\R=a^{m+1}\R=p\R,$ $\R a^m\subseteq \R a^{m+1}$ and $\R\left(a^m\right)^{*}\subseteq \R p.$
\item $a^{m+1}\in\R^{(1)},~\fpower{\circ}p=\fpower{\circ}\left(a^m\right)=\fpower{\circ}\left(a^{m+1}\right),~(a^{m+1})\fpower{\circ}\subseteq \left(a^m \right)\fpower{\circ}$, and $p\fpower{\circ}\subseteq \left(\left( a^m\right)^* \right)\fpower{\circ}.$
\end{enumerate}
If the previous assertions hold true, then assertions (ii) and (iii) {give the same} unique idempotent $p.$
Furthermore, $a^m\left(a^{m+1} \right)^{(1)}p$ is invariant under the choice of $\left(a^{m+1} \right)^{(1)}\in a^{m+1}\{1\}$
and $a^{\wc}=a^m\left(a^{m+1} \right)^{(1)}p.$
\end{theorem}
\begin{proof}

(i) $\Rightarrow$ (ii):\hspace{0.3cm}
Let $a\in\R^{\wc}$ with $\ind(a)\leq m$  and $p=aa^{\wc}.$ Then we obtain
\begin{align*}
a^m&=a^{\wc} a^{m+1}=a \left(a^{\wc} \right)^2 a^{m+1}\Longrightarrow \R a^m\subseteq \R a^{m+1},\; a^m\R\subseteq p\R,\\
p&=aa^{\wc}=a^m\left(a^{\wc} \right)^m=a^{m+1}\left(a^{\wc} \right)^{m+1}\Longrightarrow p\R\subseteq a^{m+1}\R\subseteq a^m\R.
\end{align*}
Therefore, $a^m\R=a^{m+1}\R=p\R$ and $\R a^m\subseteq \R a^{m+1}.$ Again, $\left(a^m\right)^{*}aa^{\wc}=\left(a^m\right)^{*}$ implies that
$\R \left(a^m \right)^*\subseteq \R p.$

(ii) $\Rightarrow$ (iii):\hspace{0.3cm}
Since, $p\R=a^{m+1}\R,$ there exists $s,t\in\R$ such that $p=a^{m+1}s$ and $a^{m+1}=pt.$
Therefore, $pa^{m+1}=p^2t=pt=a^{m+1}.$ Hence, $a^{m+1}sa^{m+1}=pa^{m+1}=a^{m+1},$ i.e., $a^{m+1}\in\R^{(1)}.$ By using Lemma \ref{lm:o}, the proof of the remaining parts follows.\\
(iii) $\Rightarrow$ (i):\hspace{0.3cm}
Let $a^{m+1}\in\R^{(1)}$. Then $\left(1-\left(a^{m+1}\right)^{(1)}a^{m+1}\right)\in \left(a^{m+1}\right)\fpower{\circ}\subseteq \left(a^{m}\right)\fpower{\circ},$ which further implies
\begin{equation}\label{eqn3.2}
   a^m\left(a^{m+1}\right)^{(1)}a^{m+1}=a^m.
\end{equation}
Using $\fpower{\circ}p=\fpower{\circ}\left(a^m\right)=\fpower{\circ}\left(a^{m+1}\right),$ we obtain
\begin{equation}\label{eqn3.3}
 (1-p)\in \fpower{\circ}p=\fpower{\circ}\left(a^m \right),~  \left(1-\left(a^{m+1}\right)^{(1)}a^{m+1}\right)\in \fpower{\circ}\left(a^{m+1}\right)=\fpower{\circ}p=\fpower{\circ}\left( a^m\right). \end{equation}
From equation  \eqref{eqn3.3}, we have
\begin{equation}\label{eqn3.4}
  pa^m=a^m, ~ a^{m+1}\left(a^{m+1}\right)^{(1)}p=p,~ \mbox{and}~  a^{m+1}\left(a^{m+1}\right)^{(1)}a^m=a^m.
\end{equation}
Let $y=a^m\left(a^{m+1}\right)^{(1)}p.$ Using equations \eqref{eqn3.2} and \eqref{eqn3.4}, we verify that
\begin{center}
$ya^{m+1}=a^m\left(a^{m+1}\right)^{(1)}pa^{m+1}=a^m\left(a^{m+1}\right)^{(1)}a^{m+1}=a^m$, and
\end{center}
\begin{center}
$ay^2= aa^m\left(a^{m+1}\right)^{(1)}pa^m\left(a^{m+1}\right)^{(1)}p=pa^m\left(a^{m+1}\right)^{(1)}p=a^m\left(a^{m+1}\right)^{(1)}p=y$.
\end{center}

Now, $p\fpower{\circ}\subseteq \left( \left(a^m\right)^{*} \right)\fpower{\circ}$ implies $(1-p)\in p\fpower{\circ}\subseteq \left(\left(a^m \right)^* \right)\fpower{\circ},$ i.e.,
$\left(a^m \right)^{*}p=\left( a^m \right)^{*}.$ Hence,
\begin{align*}
\left( a^m \right)^{*}ay=\left( a^m \right)^{*}a^{m+1}\left(a^{m+1}\right)^{(1)}p=\left( a^m \right)^{*}p=\left( a^m \right)^{*}.
\end{align*}
Thus, $a^{\wc}=y=a^m\left(a^{m+1}\right)^{(1)}p.$ Using equation \eqref{eqn3.2}, we obtain
\begin{align*}
a^m\left(a^{m+1}\right)^{(1)}p&=a^m \left(a^{m+1}\right)^{(1)}  a^{m+1} \left(a^{m+1}\right)^{(1)}p.
\end{align*}
Next, we claim that idempotent $p$ is unique. Suppose that there exist two idempotents $p_1,p_2\in\R$ satisfying (ii) and (iii). Then we obtain
$$
p_1\R=a^m\R=p_2\R,\; p_1\fpower{\circ}\subseteq \left(\left(a^m \right)^{*} \right)\fpower{\circ} \mbox{ and } p_2\fpower{\circ}\subseteq \left(\left(a^m \right)^{*} \right)\fpower{\circ}.
$$
There exist $u,v\in\R$ such that $p_1=a^m u$ and $p_2=a^m v.$ Since, $\left(a^m \right)^{*}p_1=\left(a^m \right)^{*}=\left(a^m \right)^{*}p_2.$
Therefore, $\left(a^m \right)^{*} a^{m}u=\left(a^m \right)^{*} a^{m}v.$ Thus, $a^mu=a^m v$ since $\R$ is proper $*$-ring. Using $a^mu=a^m v$, we obtain
$(u-v)\in \left(a^{m+1}\right)^{0}\subseteq   \left(a^{m}\right)^{0}$. Thus, $a^mu=a^mv$. Now,
$p_1=a^mu=a^mv=p_2$, and hence this completes the proof.
\end{proof}

An equivalent condition for the existence of the weak core inverse is discussed in the next result.

\begin{theorem}\label{prop:ind}
Let $a,z\in\R$. For $m,n\in\N$, if
$$
za^{m+1}=a^{m},~~ az^2=z,~~\left(a^{n} \right)^{*}az=\left(a^{n} \right)^{*},
$$
then  $a\in\R^{\wc}$.
\end{theorem}
\begin{proof}
It is sufficient to show only $(a^m)^*az=(a^m)^*$. Using the given hypothesis and Lemma \ref{lm:Drazin}, we obtain
\begin{equation*}
\begin{split}
\left(a^{m} \right)^{*}az&=\left(za^{m+1} \right)^{*}az=\left(az^2a^{m+1} \right)^{*}az=\left(a^nz^{n+1}a^{m+1} \right)^{*}az\\
&=\left(z^{n+1}a^{m+1} \right)^{*}\left(a^{n} \right)^{*}az=\left(z^{n+1}a^{m+1} \right)^{*}\left(a^{n} \right)^{*}\\
&=\left(a^{n}z^{n+1}a^{m+1} \right)^{*}=\left(az^2a^{m+1} \right)^{*}=\left(za^{m+1} \right)^{*}\\
&=\left(a^{m} \right)^{*}.
\qedhere
\end{split}
\end{equation*}
\end{proof}
The existence of the Drazin inverse {through the weak core inverse} is discussed in the following proposition.
\begin{proposition}
Let $a\in\R^{\wc}$ with $\ind_{wc}(a)=k$. Then $a\in\R^{D}$ with  $\ii(a)=k$.
\end{proposition}
\begin{proof}
Let $y=a^{\wc}$. Then by Lemma \ref{lm:Drazin}, $a\in\R^{D}$ and $a^{D}=y^{k+1}a^k$ with $\ii(a)\leq k.$ Next we will claim that $\ii(a)= k.$  Suppose $\ii(a)<k$.  Now
\begin{align*}
a^{k-1}&=a^{D}a^{k}=y^{k+1}a^ka^k=y^k\left(ya^{k+1}\right)a^{k-1}\\
&=y^k a^k a^{k-1} = y^{k-1} \left(ya^{k+1}\right)a^{k-2} =y^{k-1} a^k a^{k-2}\\
&=\ldots=y^2 a^k a=ya^k.
\end{align*}
Using Definition \ref{def-wc} and Theorem \ref{prop:ind}, we have $\ind_{wc}(a)\leq k-1,$ which contradicts the hypothesis.
Hence, $\ii(a)=k$.
\end{proof}

{In case of the Moore-Penrose inverse, we have a well-known identity $(a^\dg)^\dg = a$  but in general, $(a^\wc)^\wc~\neq~a$,  we present an example which shows this
 fact.}
\begin{example}\rm
Let $\R=\mathcal{M}_3(\mathbb{R})$ and
$
A=    \begin{pmatrix}
     0      &        8     &        -8\\
       8     &        -5  &            8\\
       8     &        -5 &             8  \\
    \end{pmatrix}\in \R$. We can find that
   $
    A^\wc=
    \begin{pmatrix}
     0    &          0      &        0\\
       0 &              1/6   &         1/6\\
    0  &              1/6      &      1/6  \\
    \end{pmatrix},~ (A^\wc)^\wc=\begin{pmatrix}
0       &       0   &   0 \\
       0  &            3/2  &          3/2 \\
       0   &           3/2 &          3/2   \\
    \end{pmatrix}$, and\\ $    ((A^\wc)^\wc)^\wc=\begin{pmatrix}
0    &          0      &        0\\
  0 &              1/6   &         1/6\\
    0  &              1/6      &      1/6  \\
       \end{pmatrix}.$
        It is clear $A \neq (A^\wc)^\wc$.
       \end{example}
{The weak core inverse of $a^{\wc}$ is always $a^2a^{\wc}$ as we show
next.}

\begin{theorem}\label{thm:wcwc}
Let $a\in\R^{\wc}$. Then  $a^{\wc}\in\R^{\wc}$ and  $\left(a^{\wc}\right)^{\wc}=a^2a^{\wc}.$
\end{theorem}
\begin{proof}
Let  $x=a^{\wc}$ and $\ind_{wc}(a)=k$. Then, we have
$$
xa^{k+1}=a^k,~~ ax^2=x,~~ \mbox{ and }~~ \left(a^k \right)^{*} a x=\left(a^k \right)^{*}.
$$
Let $y=a^2x.$ Then by Lemma \ref{lm:Drazin}, we obtain
\begin{align*}
yx^{k+1}&=a^2x^{k+2}=axx^{k}=ax^2x^{k-1}=x^k,\\
xy^2&=xa^2xa^2x=xa^2xa^{k+1}x^{k}=x a^2 \left( a^{k}x^{k}\right)=x a^{k+1}ax^k =a^{k+1}x^k=a^2x=y,\\
\left(a^k \right)^{*} xy &= \left(a^k \right)^{*} x a^2 x=\left(a^k \right)^{*}xa^{k+1}x^{k}=\left(a^k \right)^{*}a^{k}x^{k}=\left(a^k \right)^{*}a x =\left(a^k \right)^{*}.
\end{align*}
Therefore, $\left(a^{\wc} \right)^{\wc}=y=a^2 x=a^2 a^{\wc}.$
\end{proof}

\begin{corollary}\label{cor:3-10}
Let $a\in\R$ be weak core invertible. Then $\left(\left( a^{\wc}\right)^{\wc}\right)^{\wc}= a^{\wc}.$
\end{corollary}
\begin{proof}
Let $b= a^{\wc}.$ Then using Theorem \ref{thm:wcwc} and Lemma \ref{lm:Drazin}, we have
\begin{equation*}
\begin{split}
\left(\left( a^{\wc}\right)^{\wc}\right)^{\wc}&=\left(b^{\wc} \right)^{\wc}=b^2b^{\wc}=\left(a^{\wc} \right)^2 \left(a^{\wc} \right)^{\wc}=\left(a^{\wc} \right)^2 \left(a^2 a^{\wc} \right)\\
&=\left(a^{\wc} \right)^2 a \left(a a^{\wc} \right)= \left(a^{\wc} \right)^2 a \left(a^k \left(a^{\wc}\right)^k \right)\\
&=a^{\wc} a^{\wc} a^{k+1} \left( a^{\wc} \right)^k = a^{\wc}   a^{k} \left( a^{\wc} \right)^k = a^{\wc} a a^{\wc} = a^{\wc}.
\qedhere
\end{split}
\end{equation*}
\end{proof}

If $a\in\R^{\wc}$, {then the weak group, the group and the weak core inverse of $a^{\wc}$ are all equal, which is proved next.}

\begin{theorem}
Let $a\in\R.$ If $a\in\R^{\wc},$ then
$\left(a^{\wc} \right)^{\textup{\textcircled{w}}}=\left(a^{\wc} \right)^{\#}=a^2a^{\wc}=\left(a^{\wc} \right)^{\wc}.$
\end{theorem}
\begin{proof}
Let $a\in\R^{\wc}$ and $\ind_{wc}(a)=k$. Then
 $a^{\wc}  a^2  a^{\wc}a^{\wc}=a^{\wc}  a  a^{\wc}=a^{\wc}$,
\begin{equation*}
  a^2 a^{\wc} a^{\wc}a^2a^{\wc}=aa^{\wc}a^2a^{\wc}=aa^{\wc}a^{k+1}(a^{\wc})^k=a^{k+1}(a^{\wc})^k= a^2  a^{\wc},~\mbox{ and}
\end{equation*}
\begin{center}
 $ a^{\wc}a^2a^{\wc}= a^{\wc}a^{k+1}\left(a^{\wc}\right)^k =  a^{k}\left(a^{\wc}\right)^k=a^2 (a^{\wc})^2=a^2a^{\wc}a^{\wc}$.
\end{center}
Thus $a^{\wc}$ is group invertible and   $\left(a^{\wc}\right)^{\#}=a^2a^{\wc}.$ Hence by Remark \ref{prop2.7} and Theorem \ref{thm:wcwc}, we obtain
\begin{equation*}
\left(a^{\wc} \right)^{\textup{\textcircled{w}}}=\left(a^{\wc} \right)^{\#}=a^2a^{\wc}.
\qedhere
\end{equation*}
\end{proof}

Using the Drazin inverse and {the} $\{1,3\}$-inverse, we can construct the weak core inverse as follows.

\begin{theorem}\label{th:wc-1-3}
Let $a\in\R^{D}$ with $\ii(a)=k$. If $(a^k)^{(1,3)}$ exists, then $a\in\R^{\wc}$. Moreover,
\begin{center}
 $a^{\wc}=a^{D}a^k\left(a^k\right)^{(1,3)}$ and  $aa^{\wc}=a^k\left(a^k\right)^{(1,3)}.$
\end{center}
 \end{theorem}
\begin{proof}
Let $y=a^{D}a^k\left(a^k\right)^{(1,3)}.$ Then
\begin{center}
$
ya^{k+1}=a^{D}a^k\left(a^k\right)^{(1,3)} a^{k+1}=a^Da^{k+1}=a^k$,
\end{center}
\begin{align*}
ay^2&=  a a^{D} a^{k} \left(a^k\right)^{(1,3)} a^{D} a^k \left(a^k\right)^{(1,3)}=a^{D} a  a^k \left(a^k\right)^{(1,3)}a^k a^{D} \left(a^k\right)^{(1,3)}\\
&{=a^{D} a  a^k a^{D} \left(a^k\right)^{(1,3)}}=a^{D} a a^{D}  a^k \left(a^k\right)^{(1,3)}= a^{D}  a^k \left(a^k\right)^{(1,3)}=y,
\end{align*}
and
\begin{align*}
\left(a^k\right)^{*}a y&=\left(a^k\right)^{*} a a^{D}a^k\left(a^k\right)^{(1,3)}= \left(a^k\right)^{*} a^{D} a  a^{k} \left(a^{k}\right)^{(1,3)}=\left(a^k\right)^{*} a^{k} \left(a^{k}\right)^{(1,3)}\\
&= \left(a^k\right)^{*} \left(a^{k} \left(a^{k}\right)^{(1,3)}\right)^*=\left(a^k(a^k)^{(1,3)} a^k\right)^*=\left(a^k\right)^{*}.
\end{align*}
Hence $a^{\wc}=a^{D}a^k\left(a^k\right)^{(1,3)}$ and $\ind_{wc}(a)\leq k$. In addition $aa^{\wc}=aa^{D}a^k\left(a^k\right)^{(1,3)}=a^k\left(a^k\right)^{(1,3)}$.
\end{proof}
\begin{corollary}\label{cor:Darzin-wc}
Let $a\in\R^{D}$ with $\ii(a)=k$. If $(a^k)^{\dagger}$ exists, then  $a^{\wc}=a^{D}a^k\left(a^k\right)^{\dg}$ and  $aa^{\wc}=a^k\left(a^k\right)^{\dag}.$
\end{corollary}
\begin{remark}
Let $a\in\R^{D}$ with $\ii(a)=k$. If $(a^k)^{(1,3)}$ exists, then $a\in\R^{\tp{\ep}}\cap\R^{\wc}$ and $a^{\wc}=a^{\tp{\ep}}$.
\end{remark}
In view of Corollary \ref{cor:Darzin-wc} and Proposition \ref{lm:drazin-group}, we have the following result.
\begin{lemma}\label{lm:Drazin-core-Dagger}
Let $a\in\R^{D}$ with $\ii(a)=k$. If $(a^k)^{\dagger}$ exists, then
\begin{center}
  $(a^{\wc})^k=(a^D)^ka^k(a^k)^{\dagger}=(a^k)^{\#}a^k(a^k)^{\dagger}=(a^k)^{\tp{\core}}$.
\end{center}
\end{lemma}
\begin{corollary}\label{cor3.20}
Let $a\in\R^{D}$ with $\ii(a)=k$. If $a^k\in\R^{\#}\cap \R^{\dagger}$, then
$
a^{\wc}=a^Da^k \left(a^{k}\right)^{\tp{\core}}.
$
\end{corollary}
\begin{proof}
Using Lemma \ref{lm:core-drazin} and  Proposition \ref{lm:drazin-group}, we obtain
\begin{align*}
a^k \left(a^{k}\right)^{\core}&=a^k \left(a^{k}\right)^{\#}  a^{k} \left(a^{k}\right)^{\dg}
= a^k \left(a^{D}\right)^{k}  a^{k} \left(a^{k}\right)^{\dg}
=aa^Da^{k} \left(a^{k}\right)^{\dg}=a^{k} \left(a^{k}\right)^{\dg}.
\end{align*}
Applying Corollary \ref{cor:Darzin-wc}, we have $a^{\wc}=a^Da^k \left(a^{k}\right)^{\core}.$
\end{proof}
The existence and construction {of the weak core inverse via the core inverse is discussed in the next result.}

\begin{theorem}\label{th:wc-c-n}
Let $a\in\R$. If $a^k\in\R^{\tp{\core}}$, then  $a\in\R^{\wc}$ and  $a^{\wc}=a^{k-1}\left(a^k \right)^{\tp{\core}}.$
\end{theorem}
\begin{proof}
Let $a^k\in\R^{\core}.$ Then
\begin{equation}\label{eqn3.5}
\left(a^k \right)^{\core} \left(a^k \right)^{2}=a^k,~~ a^k \left(\left(a^k \right)^{\core}\right)^2=\left(a^k \right)^{\core} \mbox{ and } \left(a^k \left(a^k \right)^{\core} \right)^{*}=a^k \left(a^k \right)^{\core}.
\end{equation}
Assume that $x=a^{k-1}\left(a^k \right)^{\core}$. Then using Proposition \ref{lm:drazin-group} and equation \eqref{eqn3.5}, we obtain
\begin{align*}
xa^{k+1}&=a^{k-1}\left(a^{k}\right)^{\core} a^{k+1}
=a^{k-1}\left(a^{k}\right)^{\#} a^{k} \left(a^{k}\right)^{\dg}a^{k+1}= a^{k-1}\left(a^{k}\right)^{\#} a^{k+1}\\
&={a^{k-1}\left(a^{D}\right)^{k} a^{k+1}}=a^k,\\
ax^2&=a\left(a^{k-1} \left(a^k\right)^{\core} \right)^2=a^k\left(a^k \right)^{\core}a^{k-1}\left(a^k \right)^{\core}
=a^k\left(a^k \right)^{\core}a^{k-1}a^k\left(\left(a^k \right)^{\core}\right)^2\\
&=a^{k-1}\left(a^k \right)^{\core}=x,\\
\left(a^k\right)^{*}ax&=\left(a^k\right)^{*} a^k \left(a^k \right)^{\core} =\left(a^k\right)^{*} \left(a^k \left(a^k \right)^{\core}\right)^{*}=\left(a^k \left(a^k \right)^{\core} a^k \right)^{*}
=  \left(a^k \right)^{*}.
\end{align*}
Hence, $a\in\R^{\wc}$ and $a^{\wc}=x=a^{k-1}\left(a^k \right)^{\core}.$
\end{proof}

{Explicit expressions for the weak core inverse and its powers are given in the following theorem.

\begin{theorem}\label{th:wc-n}
Let $a\in \R^{\wc}$. Then $a^n\in\R^{\wc}$ and $\left(a^{\wc}\right)^n=(a^n)^{\wc}$ for all $n\geq 1.$
Moreover, $a^{\wc}=a^{n-1}\left(a^n \right)^{\wc}.$
\end{theorem}
\begin{proof}
Let $a\in\R^{\wc}$ and $\ind_{wc}(a)=m$. Setting $y=\left(a^{\wc}\right)^n,$ we have
\begin{align*}
y\left(a^n\right)^{m+1}&=\left(a^{\wc}\right)^{n} \left(a^n\right)^{m+1}= \left(a^{\wc}\right)^{n-1}a^{\wc}a^{m+1} a^{(n-1)(m+1)}\\
&=\left(a^{\wc}\right)^{n-1}a^{m} a^{(n-1)(m+1)}=\left(a^{\wc}\right)^{n-2}a^{\wc}a^{m+1}a^{m} a^{(n-2)(m+1)}\\
&=\left(a^{\wc}\right)^{n-2}\left(a^{m}\right)^2 a^{(n-2)(m+1)}=\cdots=a^{\wc}\left(a^{m}\right)^{n-1} a^{m+1}\\
&=a^{\wc}a^{m+1}\left(a^m \right)^{n-1}=a^m \left(a^m\right)^{n-1}=\left(a^m\right)^{n}\\
&=\left(a^n\right)^{m},
\end{align*}
and
\begin{center}
$
a^ny^2=a^n\left(a^{\wc}\right)^{n} \left(a^{\wc}\right)^{n}=aa^{\wc}\left(a^{\wc}\right)^{n}=a\left(a^{\wc}\right)^2 \left(a^{\wc}\right)^{n-1}=\left(a^{\wc}\right)^{n}=y.$
\end{center}
Since $nm\geq m,$ applying Lemma \ref{lm:Drazin}, we have
\begin{align*}
\left(a^{nm}\right)^{*} a^n y&= \left(a^{nm}\right)^{*} a^n \left(a^{\wc}\right)^n= \left(a^{n m}\right)^{*} a a^{\wc} = \left(a^{nm-m}\right)^{*}\left(a^m\right)^{*} a a^{\wc}\\
&=\left(a^{nm-m}\right)^{*}\left(a^m\right)^{*}=\left(a^{nm}\right)^{*}.
\end{align*}
Hence by Definition \ref{def-wc}, we claim that $a^n\in\R^{\wc}$ and $\left(a^n\right)^{\wc}=y=\left(a^{\wc}\right)^n.$
Conversely, let $\ind_{wc}\left(a^n\right)=l$ and $z=\left( a^n\right)^{\wc}.$ Then by Definition \ref{def-wc}, we have
$$
z\left(a^n \right)^{l+1}=\left(a^n \right)^{l},~~ a^nz^2=z,~~ \mbox{ and } \left(a^{nl}\right)^{*}a^nz=\left(a^{nl}\right)^{*}.
$$
Suppose that $x=a^{n-1}z.$ Now, we have
\begin{align*}
xa^{nl+1}&=a^{n-1}za^{nl+1}=a^{n-1}a^nz^2a^{nl+1}=a^{n-1}a^nzza^{nl+1}\\
&=a^{n-1}\left(a^{2n}zz^2 \right) a^{nl+1}=\cdots=a^{n-1}\left(a^{ln}z z^l \right)a^{nl+1}\\
&=a^{nl+n-1}z^{l+1}a^{nl}a=a^{nl+n-1}z^{l+1}\left(a^n\right)^l a\\
&=a^{nl+n-1}\left(a^n\right)^{D}a = \left( a^n\right)^{D}a^{nl+n}=\left( a^n\right)^{D} \left( a^n\right)^{l+1}=(a^n)^{l}\\
&=a^{nl},
\end{align*}
\begin{align*}
ax^2&=aa^{n-1} z a^{n-1}z=a^nza^{n-1}z=a^n z a^{n-1} \left(a^n z^2\right)\\
&=a^n z a^{n-1} \left(\left(a^n \right)^{l+1} z^{l+2} \right)=a^n \left(z \left(a^{n}\right)^{l+1} \right)a^{n-1}z^{l+2}\\
&=a^{n} a^{nl} a^{n-1} z^{l+2} = a^{n-1} \left( \left(a^n \right)^{l+1}z^{l+2}\right)\\
&=a^{n-1}z=x,
\end{align*}
and
\begin{align*}
\left(a^{nl} \right)^{*} a x=\left(a^{nl} \right)^{*}a^n z=\left(a^{nl} \right)^{*}.
\end{align*}
Hence, $a^{\wc}=x=a^{n-1}z=a^{n-1}\left( a^n\right)^{\wc}.$
\end{proof}
\begin{remark}\label{rem3.18}
The above theorem need not be true in general if we use two different elements $a$ and $b$ in $\R^{\wc}$, i.e., $(ab)^{\wc} \neq a^\wc b^\wc,$ when  $a\neq b$.
\end{remark}
In support of the Remark \ref{rem3.18}, the following example is worked-out.
\begin{example}\label{weakrevEx}\rm
Let $\R=\mathcal{M}_3(\mathbb{R})$. Clearly $\R$ is a {proper $*$-ring with transpose as involution $*$.} Consider $A=
    \begin{pmatrix}
-3 &            -3           &  -1\\
1   &           1            &  1   \\
0    &          0            &  0    \\
    \end{pmatrix}$ and $B=
    \begin{pmatrix}
    3 & 1 & 0\\
    -3 & -1 & 0\\
    2 & -2 & 0\\
    \end{pmatrix}$. We can verify that
    \begin{equation*}
    A^\wc=\begin{pmatrix}
-9/20  &         3/20    &       0       \\
       3/20     &     -1/20   &        0   \\
       0    &          0     &         0  \\
    \end{pmatrix}, ~~~and~~~
    B^\wc=\begin{pmatrix}
   1/12     &     -1/12     &      1/6     \\
      -1/12     &      1/12  &        -1/6   \\
       1/6      &     -1/6    & 1/3  \\
       \end{pmatrix}
\end{equation*}
 are respectively the weak inverse of $A$ and $B$.
Also we can see that
$$
    \begin{pmatrix}
    -1/8        &    1/8        &    0    \\
       1/8      &     -1/8       &     0    \\
       0         &     0        &      0  \\
    \end{pmatrix} = (A B)^\wc \neq A^\wc~B^\wc =
    \begin{pmatrix}
 -1/20  &         1/20 &         -1/10  \\
       1/60  &        -1/60  &         1/30 \\
       0    &          0      &        0      \\
    \end{pmatrix}.$$
           \end{example}

The additive property, $(a+b)^\wc \neq a^\wc + b^\wc$ for {the weak core inverse does not hold in general, as shown in the next example.}

\begin{example}\rm
Let $A$ and $B$ {be} defined as in Example \ref{weakrevEx}. We can see that
\begin{equation*}
\begin{pmatrix}
    -1/4    &       -1/4     &       1/4     \\
      -1/4      &     -1/4    &       -1/4  \\
      -1/2      &      1/2     &       1/2\\
    \end{pmatrix} =(A + B)^\wc \neq  A^\wc  + B^\wc=
    \begin{pmatrix}
   -11/30   &        1/15       &    1/6     \\
       1/15  &         1/30     &     -1/6     \\
       1/6    &       -1/6      &      1/3    \\
    \end{pmatrix}
    \end{equation*}
\end{example}

Now we discuss a few sufficient conditions for the additive property.

\begin{theorem}\label{th:sum-wc}
Let $a,b\in\R^{\wc}$ with $ab=0=ba$ and $a^{*}b=0.$ Then $(a+b)^{\wc}=a^{\wc}+b^{\wc}.$
\end{theorem}
\begin{proof}
Suppose that $ab=0=ba$ and $a^{*}b=0=\left(a^{*}b \right)^{*}=b^{*}a.$ Using these hypotheses, we have
\begin{align*}
ab^{\wc}&=ab\left(b^{\wc}\right)^2=0,\\
ba^{\wc}&=ba\left(a^{\wc}\right)^2=0,\\
b^{\wc}a&=b^{\wc}bb^{\wc}a=b^{\wc} \left(b^{\wc} \right)^{*}b^{*}a=0,\\
a^{\wc}b&=a^{\wc}aa^{\wc}b=a^{\wc}\left(a^{\wc} \right)^{*}a^{*}b=0,\\
a^{\wc}b^{\wc}&=a^{\wc}\left(a^{\wc} \right)^{*}a^{*}b\left(b^{\wc}\right)^2=0,\\
b^{\wc}a^{\wc}&=b^{\wc}\left(b^{\wc} \right)^{*}b^{*}a\left(a^{\wc}\right)^2=0.
\end{align*}
Let $\ind_{wc}(a)=k_1$, $\ind_{wc}(b)=k_2$ and $k=\max(k_1,k_2)$. Using Lemma \ref{lm:Drazin}, we obtain
$$
a^k\left(a^{\wc} \right)^k a^k=a^k~~\mbox{ and }~~b^k\left(b^{\wc} \right)^k b^k=b^k.
$$
Now, we have
\begin{align*}
&(a+b)^k \left( \left(a^{\wc}\right)^k +\left(b^{\wc}\right)^k \right)(a+b)^k \\
&=\left(a^k+b^k\right) \left( \left(a^{\wc}\right)^k +\left(b^{\wc}\right)^k \right) \left(a^k+b^k\right)
=\left(a^k\left(a^{\wc}\right)^k+b^k \left(b^{\wc}\right)^k \right) \left(a^k+b^k\right)\\
&=\left(aa^{\wc}+bb^{\wc}\right) \left(a^k+b^k\right)
=aa^{\wc}a^k+bb^{\wc}b^k=a^k \left(a^{\wc}\right)^k a^k + b^k \left(b^{\wc}\right)^k b^k\\
&=a^k+b^k,
\end{align*}
and
\begin{align*}
\left((a+b)^k \left( \left(a^{\wc}\right)^k +\left(b^{\wc}\right)^k \right) \right)^{*}
&=\left(aa^{\wc}+bb^{\wc}\right)^{*}=\left( aa^{\wc}\right)^{*}+\left( bb^{\wc}\right)^{*}=aa^{\wc}+bb^{\wc}\\
&= (a+b)^k \left( \left(a^{\wc}\right)^k +\left(b^{\wc}\right)^k \right).
\end{align*}
Therefore, $ \left(a^{\wc}\right)^k +\left(b^{\wc}\right)^k$  is $\{1,3\}$ inverse of $(a+b)^k.$ Using Lemma \ref{prop:sum-Drazin}, Theorem \ref{th:wc-1-3}
and Corollary \ref{cor3.20}, we have
\begin{equation*}
\begin{split}
(a+b)^{\wc}&=(a+b)^{D} (a+b)^{k} \left( \left(a^{\wc}\right)^k +\left(b^{\wc}\right)^k \right)\\
&=\left(a^{D}+b^{D}\right)\left(a^{k}+b^{k}\right)\left( \left(a^{\wc}\right)^k +\left(b^{\wc}\right)^k \right)\\
&=\left( a^D a^k+b^Db^k \right) \left( \left(a^{\wc}\right)^k +\left(b^{\wc}\right)^k \right)=a^D a^k \left(a^{\wc}\right)^k+b^D b^k \left(b^{\wc}\right)^k\\
&= a^D a^k \left( a^k \right)^{\core}+ b^D b^k \left( b^k \right)^{\core}\\
&=a^{\wc}+b^{\wc}.
\qedhere
\end{split}
\end{equation*}
\end{proof}

\section{Central weak core inverse}
In this section, we introduce a central weak core inverse in a proper $*$-ring $\R$. Several characterizations of this inverse and its relation to the generalized inverses previously introduced are presented. This section begins with the following definitions.

\begin{definition}\label{def:cwep}
Let $a\in\R$. An element $x\in\R$ satisfying
$$
ax\in C(\R),~~xa^{k+1}=a^k, ~~xax=x,~(ax)^*=(ax)~~\mbox{ for some }k\geq1,
$$
is called {a central weak core inverse of $a$, and it is denoted by $a^{\cc}.$ The smallest positive integer $k$ satisfying the above equations is called the index}
(central weak core index) of $a$ and  is denoted by $\ind_{cw}(a)$.
\end{definition}
We denote the set of central weak core invertible elements in $\R$ by $\R^{\cc}$.  and discuss the basic properties of the central weak core inverse.
\begin{proposition}\label{prop4.2}
Let $a\in\R$ be central weak core invertible and $x=a^{\cc}$. Then the following assertions hold:
\begin{enumerate}[\rm(i)]
    \item $ax^2=x$;
    \item $ax=xa$;
    \item $x^2a=x$;
    \item $xa^2x=ax$.
    \end{enumerate}
\end{proposition}

\begin{proof}
(i) Let $x=a^{\cc}$. Using the centrality of $ax$, we obtain $x=xax=axx=ax^2$. \\
(ii) {From Definition} \ref{def:cwep} and Lemma \ref{lm:Drazin}, we have
\begin{eqnarray*}
ax&=&a(xax)=a(ax)x=a^2x^2=\cdots=a^kx^k=xa^{k+1}x^{x}=\cdots=xa^2x\\
&=& xa(ax)=x(ax)a=xa.
\end{eqnarray*}
(iii) $x^2a=xxa=xax=x.$\\
(iv) Using (ii), we obtain $xa^2x=xaax=axax=ax.$
\end{proof}
The uniqueness of the central weak core inverse is proven in the next result.
\begin{theorem}
Let $a\in\R^{\cc}$. Then the central weak core inverse of $a$ is unique.
\end{theorem}
\begin{proof}
Suppose that there exist two inverses,  $x$ and $y$. Then, by Lemma \ref{lm:Drazin}, we obtain
\begin{equation*}
\begin{split}
x&=xax=xa^{k+1}x^{k+1}=a^kx^{k+1}=ya^{k+1}x^{k+1}=yax=axy\\
&=xay=xa^{k+1}y^{k+1}=a^ky^{k+1}=ya^{k+1}y^{k+1}=yay=y.
\qedhere
\end{split}
\end{equation*}
\end{proof}

In view of Proposition \ref{prop4.2}, the following results can be easily verified.
\begin{theorem}
If $a\in\R$ is central weak core invertible, then
\begin{enumerate}[\rm(i)]
\item  $a$ is core-EP invertible and $a^{\tp{\ep}}=a^{\cc}$;
\item  $a$ is central Drazin-invertible $a^{\tp{\cd}}=a^{\cc}$;
\item $a$ is Drazin-invertible $a^{D}=a^{\cc}$.
\end{enumerate}
\end{theorem}
In a special case we can easily prove the following result for $k=1$.
\begin{proposition}\label{prop4.3}
Let $a\in\R^{{\cc}}$ and $\ind_{cw}(a)=1$. Then
\begin{enumerate}[\rm(i)]
    \item $a\in\R^{\#}\cap\R^{\tp{\core}}\cap\R^{\dagger}$;
    \item $a^{\#}=a^{\tp{\core}}=a^{\dagger}=a^{\cc}$;
    \item $a$ is the EP element.
\end{enumerate}
\end{proposition}
The following results provide us with the necessary and sufficient condition for an element $a\in\R^{\cc}$ to be core-EP invertible.

\begin{proposition}\label{prop-cc-wc-iff}
Let $a\in \R$. Then
$a\in\R^{\cc}$ if and only if $a\in\R^{\tp{\ep}}$ and  $aa^{\tp{\ep}}$ is central.
\end{proposition}

The proof follows from the definitions of the central weak core inverse and core-EP inverse.

\begin{theorem}
Every core-EP invertible element of $\R$ is central weak core invertible if and only if $\R$ is abelian.
\end{theorem}
\begin{proof}
If $\R$ is abelian, then every weak core element is also a central weak core element.

Conversely, assume that $a\in\R^{\ep}\subseteq \R^{\cc}.$ Now wenow prove that $\R$ is abelian.
Suppose that $\R$ is not an abelian. Then $aa^{\cc}\neq a^{\cc} a.$ This shows that $a\notin \R^{\cc},$
which is a contraction.
\end{proof}

{A few characterizations} of the central weak core inverse are presented in the following results.
\begin{lemma}\label{prop:am=bm}
Let $a\in\R$ be the central weak core invertible with $\ind_{cw}(a)=k$. For $m\geq k$ and $b\in\R$, if $a^mb\in C(\R)$ or $ba^m\in C(\R)$, then $a^mb=ba^m.$
\end{lemma}
\begin{proof}
Let $a^mb\in C(\R)$. Then by centrality of $aa^{\cc}$, we have
\begin{align*}
ba^m&=ba^{\cc}a^{m+1}=baa^{\cc}a^{m}=aa^{\cc}ba^{m}
= a^m(a^{\cc})^m b a^m= (a^{\cc})^m (a^mb)a^m\\
&=(a^{\cc})^m a^m(a^mb)=a^{\cc}a a^m b=a^m b.
\end{align*}
Similarly, we can show that if $ba^m\in C(\R)$ then $a^mb=ba^m.$
\end{proof}

\begin{theorem}
Let $a\in\R$ be central weak core invertible with $\ind_{cw}(a)=k$. Then
\begin{center}
  $\fpower{\circ}\left(a^m\right)=\left(a^m\right)\fpower{\circ}=\fpower{\circ}\left(a^{\cc}\right)=\left(a^{\cc}\right)\fpower{\circ}$ for any integer $m\geq k.$
\end{center}
\end{theorem}
\begin{proof}
Let $b\in \left(a^m \right)\fpower{\circ}$. Then $a^mb=0\in C(\R).$ Hence, $a^{\cc}b=\left(\left(a^{\cc}\right)^{m+1}a^m \right)b=\left(a^{\cc}\right)^{m+1}\left(a^mb \right)=0.$
Using Lemma \ref{prop:am=bm}, we have $ba^m=a^mb=0,$ which yields $\left(a^m\right)\fpower{\circ}\subseteq \left(a^{\cc}\right)\fpower{\circ}$ and $\left(a^m\right)\fpower{\circ}\subseteq \fpower{\circ}\left(a^{m}\right).$ Similarly, it can be verified that
$\left(a^{\cc}\right)\fpower{\circ} \subseteq \left(a^m\right)\fpower{\circ}$ and $\fpower{\circ}\left(a^{m}\right) \subseteq \left(a^m\right)\fpower{\circ}.$ Thus, $\fpower{\circ}\left(a^{m}\right)=\left(a^{m}\right)\fpower{\circ}=\left(a^{\cc}\right)\fpower{\circ}.$ Since $a^m=a^{\cc}a^{m+1}$ and $a^{\cc}=a(a^{\cc})^2=a^m(a^{\cc})^{m+1}$, it follows that
$\fpower{\circ}\left(a^{\cc}\right) =\fpower{\circ}\left(a^{m}\right).$ Hence completes the proof.
\end{proof}

{
Likewise, the weak core inverse and  central weak core inverse do not satisfy the relation
 $(a^{\cc})^{\cc}=a$ for all $a\in\R$.}

\begin{theorem}\label{th:cc-cc}
Let $a\in\R^{\cc}.$ Then $a^{\cc}\in\R^{\cc}.$ In particular,
$\left(a^{\cc} \right)^{\cc}=a^2a^{\cc}$.
\end{theorem}
\begin{proof}
Let $x=a^{\cc}$ and $\ind_{cw}(a)=k.$ From  $ax\in C(\R)$, we have $tax=axt$ for every  $t\in\R$.  Let $y=a^2x$. Then using Proposition \ref{prop4.2}, we obtain
\begin{align*}
txy&=t\left(xa^2x \right)=t(ax)=(ax)t=\left(xa^2x \right)t=xyt.
\end{align*}
Thus, $xy\in C(\R).$ Now, we have
\begin{align*}
yxy&=a^2xxa^2x=a^2\left( x^2 a\right)ax=a^2xax=a^2x=y,\\
(xy)^{*}&=\left( xa^2x\right)^{*}=(ax)^{*}=ax=xa^2x=xy.
\end{align*}
{Following the technique in the proof of } Theorem \ref{thm:wcwc}, we can show that
$yx^{k+1}=x^k.$
Hence,
\begin{equation*}
 (a^{\cc})^{\cc}=x^{\cc}=y=a^2a^{\cc}.
 \qedhere
\end{equation*}
\end{proof}
{Using Corollary \ref{cor:3-10}, one} can prove the following result.
\begin{corollary}
Let $a\in\R^{\cc}.$ Then $\left(\left(a^{\cc} \right)^{\cc}\right)^{\cc}=a^{\cc}$.
\end{corollary}

{The powers of the central weak core inverse and central weak core inverse of the power of an element $a \in \R$ can be interchanged.}
\begin{theorem}
Let $a\in\R^{\cc}.$ Then $a^n\in\R^{\cc}$ and
$\left(a^n\right)^{\cc}=\left( a^{\cc}\right)^n$ for any positive integer $n$.
\end{theorem}
\begin{proof}
Let  $a\in\R^{\cc}$ with $\ind_{cw}(a)=m$ and $y=\left(a^{\cc}\right)^n$. Since $aa^{\cc}\in C(\R)$, it follows that $taa^{\cc}=aa^{\cc}t$ for all $t\in\R$. Using Proposition \ref{prop4.2} and Lemma \ref{lm:Drazin}, we have
\begin{align*}
ta^n \left(a^{\cc} \right)^{n}&=t\left(aa^{\cc}\right)^n=\left(aa^{\cc}\right)^nt=a^n\left(a^{\cc} \right)^{n}t.
\end{align*}
Hence, $a^n \left(a^{\cc} \right)^{n}\in C(\R).$ Further,
\begin{align*}
ya^ny&=\left(a^{\cc} \right)^n a^n \left(a^{\cc} \right)^n=\left(a^{\cc} \right)^n a a^{\cc}=\left(a^{\cc} \left(a^{\cc}\right)^n\right)^n=y,\\
\left(a^n y \right)^{*}&=\left(a^n \left(a^{\cc}\right)^n \right)^{*}=\left(aa^{\cc}\right)^{*}=aa^{\cc}=a^n \left(a^{\cc}\right)^n=a^ny.
\end{align*}
Using the proof of Theorem \ref{th:wc-n}, we can establish that
$ y\left(a^n \right)^{m+1}=\left(a^n\right)^m,$ which proves the theorem.
\end{proof}

{The construction of the central weak core inverse via the $\{1,3\}$-inverse} and the central Drazin inverse is discussed in the following result.

\begin{theorem} \label{th:cc-1-3}
Let $a\in\R$ be the central Drazin-invertible with the Drazin index $\ii(a)=k.$ If $(a^k)^{(1,3)}$ exists, then $a^{\cc}=a^{\textup{\textcircled{d}}}a^k\left(a^k\right)^{(1,3)}$. Moreover, $aa^{\cc}=a^k\left(a^k\right)^{(1,3)}.$
\end{theorem}
\begin{proof}
Let $y=a^{\textup{\textcircled{d}}}a^k\left(a^k\right)^{(1,3)}.$ Then, using a technique  similar to Theorem \ref{th:wc-1-3}, we obtain
$ya^{k+1}=a^k$.
Next, we will claim that $ay\in C(\R)$. Using the centrality of $aa^{\textup{\textcircled{d}}}$, we obtain
\begin{equation*}
\begin{split}
ay&=aa^{\tp{\textcircled{d}}}a^k\left(a^k\right)^{(1,3)}=a^k\left(a^k\right)^{(1,3)}aa^{\tp{\textcircled{d}}}
=a^k\left(a^k\right)^{(1,3)}a^k\left(a^{\tp{\textcircled{d}}}\right)^k\\
&=a^k\left(a^{\tp{\textcircled{d}}}\right)^k=aa^{\tp{\textcircled{d}}}\in C(\R).
\end{split}
\end{equation*}
Again, using Proposition \ref{prop2.11}, we have
\begin{equation*}
\begin{split}
yay&=ay^2=aa^{\tp{\textcircled{d}}}a^{\textup{\textcircled{d}}}a^k\left(a^k\right)^{(1,3)}=a^{\tp{\textcircled{d}}}aa^{\textup{\textcircled{d}}}a^k\left(a^k\right)^{(1,3)}= a^{\textup{\textcircled{d}}}a^k\left(a^k\right)^{(1,3)}=y,\\
(ay)^{*}&=\left( a a^{\textup{\textcircled{d}}}a^k\left(a^k\right)^{(1,3)} \right)^{*}
=\left( a^k\left(a^k\right)^{(1,3)} \right)^{*}=a^k\left(a^k\right)^{(1,3)}\\
&=a a^{\textup{\textcircled{d}}}a^k\left(a^k\right)^{(1,3)}=ay.
\qedhere
\end{split}
\end{equation*}
Moreover, $aa^{\cc}=aa^{\tp{\cd}}a^k\left(a^k\right)^{(1,3)}=a^{\tp{\cd}}a^{k+1}\left(a^k\right)^{(1,3)}=a^k\left(a^k\right)^{(1,3)}.$
\end{proof}
\begin{corollary}\label{prop4.16}
Let $a\in\R^{\tp{\cd}}$ with $\ind_{cd}(a)=k$. If $(a^k)^{\dagger}$ exists, then  $a^{\cc}=a^{\tp{\cd}}a^k\left(a^k\right)^{\dg}$ and  $aa^{\cc}=a^k\left(a^k\right)^{\dag}.$
\end{corollary}
From Corollary \ref{prop4.16}, we can derive the following result.
\begin{lemma}\label{lemma4.17}
Let $a\in\R^{\tp{\cd}}$ with $\ind_{cd}(a)=k$. If $(a^k)^{\dagger}$ exists, then  $(a^{\cc})^k=(a^{\tp{\cd}})^ka^k(a^k)^{\dagger}=(a^k)^{\#}a^k(a^k)^{\dagger}=(a^k)^{\tp{\core}}$.
\end{lemma}

\begin{corollary}\label{cor4.18}
Let $a\in\R^{\tp{\cd}}$ with $\ind_{cd}(a)=k$. If $a^k\in\R^{\#}\cap \R^{\dagger}$, then
$
a^{\cc}=a^{\tp{\cd}}a^k \left(a^{k}\right)^{\tp{\core}}.
$
\end{corollary}
\begin{proof}
Using Lemma \ref{lm:core-drazin} and  Lemma \ref{lm:central-D-group}, we obtain
\begin{align*}
a^k \left(a^{k}\right)^{\core}&=a^k \left(a^{k}\right)^{\#}  a^{k} \left(a^{k}\right)^{\dg}
= a^k \left(a^{\tp{\cd}}\right)^{k}  a^{k} \left(a^{k}\right)^{\dg}
=aa^{\tp{\cd}}a^{k} \left(a^{k}\right)^{\dg}=a^{k} \left(a^{k}\right)^{\dg}.
\end{align*}
Applying Corollary \ref{prop4.16}, we have $a^{\cc}=a^{\tp{\cd}}a^k \left(a^{k}\right)^{\core}.$
\end{proof}
Next, we discuss the additive property of {the}  central Drazin inverse, which is essential for proving the additive property of {the} central weak core inverse.

\begin{lemma}\label{lem4.13}
Let $a,b\in\R^{\tp{\textcircled{d}}}$ with $ab=0=ba$. Then $(a+b)\in\R^{\tp{\textcircled{d}}}$ and  $(a+b)^{\tp{\textcircled{d}}}=a^{\tp{\textcircled{d}}}+b^{\tp{\textcircled{d}}}.$
\end{lemma}
\begin{proof}
Let  $ab=0=ba$. Then, using Proposition \ref{prop2.11}, we can easily obtain $a^{\tp{\textcircled{d}}}b=0=ba^{\tp{\textcircled{d}}}$ and $ab^{\tp{\textcircled{d}}}=0=b^{\tp{\textcircled{d}}}a$. Now
\begin{center}
    $(a+b)(a^{\tp{\textcircled{d}}}+b^{\tp{\textcircled{d}}})=aa^{\tp{\textcircled{d}}}+bb^{\tp{\textcircled{d}}}=a^{\tp{\textcircled{d}}}a+b^{\tp{\textcircled{d}}}b+a^{\tp{\textcircled{d}}}b+b^{\tp{\textcircled{d}}}a=(a^{\tp{\textcircled{d}}}+b^{\tp{\textcircled{d}}})(a+b)$,
\end{center}
\begin{equation*}
    (a^{\tp{\textcircled{d}}}+b^{\tp{\textcircled{d}}})(a+b)(a^{\tp{\textcircled{d}}}+b^{\tp{\textcircled{d}}})=a^{\tp{\textcircled{d}}}aa^{\tp{\textcircled{d}}}+b^{\tp{\textcircled{d}}}bb^{\tp{\textcircled{d}}}=a^{\tp{\textcircled{d}}}+b^{\tp{\textcircled{d}}}, \mbox{ and }
\end{equation*}
\begin{center}
    $(a^{\tp{\textcircled{d}}}+b^{\tp{\textcircled{d}}})(a+b)=a^{\tp{\textcircled{d}}}a+b^{\tp{\textcircled{d}}}b\in C(\R)$.
\end{center}
Hence $(a+b)\in\R^{\tp{\textcircled{d}}}$ and  $(a+b)^{\tp{\textcircled{d}}}=a^{\tp{\textcircled{d}}}+b^{\tp{\textcircled{d}}}.$
\end{proof}
{In view of Theorem \ref{th:cc-1-3}, Lemma \ref{lemma4.17}, Lemma \ref{lem4.13}, and Theorem \ref{th:sum-wc}, the following} result can be established.
\begin{theorem}
Let $a,b\in\R^{\cc}$ with $ab=0=ba$ and $a^{*}b=0.$ Then $(a+b)\in\R^{\cc}$ and $(a+b)^{\cc}=a^{\cc}+b^{\cc}.$
\end{theorem}

\section{Conclusion}
{
We have presented the notions of weak core and central weak core inverses in a proper $*$-ring. Using these concepts, several characterizations  have been established in connection to other generalized inverses. The additive properties of these inverses are demonstrated. A few numerical examples are provided to validate some of our claims and remarks.  We pose the following problems for further research, which have not been addressed in this paper.}

\begin{itemize}
    \item[$\bullet$] It would be interesting to investigate the reverse order law for these classes of inverses (see Remark \ref{rem3.18}).
    \item[$\bullet$] To study these classes of inverses in the framework of complex matrices and tensors.
\item[$\bullet$] To establish weighted weak core and central weak core inverses.
    \item[$\bullet$] {Partial ordering for these inverses would be interesting to study}.
   \item[$\bullet$]  {To explore these classes of inverses in the framework of quaternions $\mathbb{H}_s$.}
 \end{itemize}

\section*{Declaration of Competing Interest}
The authors declare that they have no known competing financial interests or personal relationships that could have appeared to influence the work reported in this paper.


\section*{Funding}
\begin{itemize}
    \item Ratikanta Behera is grateful for the supported by Science and Engineering Research Board (SERB), Department of Science and Technology, India, under the Grant No. EEQ/2022/001065. 

\item Jajati Keshari Sahoo is grateful for the supported by Science and Engineering Research Board (SERB), Department of Science and Technology, India, under the Grant No. SUR/2022/004357. 

\item Ram N. Mohapatra is grateful to the Mohapatra Family Foundation and the College of Graduate Studies, University of Central Florida, Orlando, for their financial support for this research.

\end{itemize}

\section*{Data Availability Statement}
Data sharing not applicable to this article as no datasets were generated or analyzed during the current study.

\bibliographystyle{acm}
\bibliography{reference}

\end{document}